\documentclass{amsart}
\usepackage{amsmath}
\usepackage{amssymb}
\usepackage{mathtools}
\usepackage{mathrsfs}
\usepackage{amsthm}
\usepackage[dvipsnames]{xcolor}
\usepackage{tikz-cd}
\usepackage{caption}

\usepackage{hyperref}
\hypersetup{
  colorlinks=true,
  breaklinks,
  linkcolor=[RGB]{51 102 204},
  filecolor=Orchid,
  urlcolor=[RGB]{51 102 204},
  citecolor=Orchid,
  linktoc=all,
}
\usepackage[nameinlink]{cleveref}

\usepackage[top=1.2in,bottom=1.2in,left=1in,right=1in]{geometry}

\theoremstyle{plain}
\newtheorem{theorem}{Theorem}[section]

\newtheorem{lemma}[theorem]{Lemma}
\newtheorem{proposition}[theorem]{Proposition}
\theoremstyle{definition}
\newtheorem{definition}[theorem]{Definition}
\newtheorem{remark}[theorem]{Remark}

\newcommand{\C}{\mathbb{C}}

\newcommand{\sB}{\mathscr{B}}
\newcommand{\Z}{\mathbb{Z}}
\newcommand{\CP}{\mathbb{CP}}
\newcommand{\para}[1]{\medskip\noindent\textbf{#1.}}

\DeclareMathOperator{\Poly}{Poly}

\DeclareMathOperator{\UConf}{UConf}
\DeclareMathOperator{\PConf}{PConf}

\DeclareMathOperator{\push}{push}
\DeclareMathOperator{\twist}{twist}
\DeclareMathOperator{\interior}{int}
\DeclareMathOperator{\im}{im}

\author{Peter Huxford}
\author{Nick Salter}

\address{Peter Huxford: Department of Mathematics, University of Chicago, 5734 S. University Ave., Chicago, IL 60637}
\email{pjhuxford@uchicago.edu}

\address{Nick Salter: Department of Mathematics, University of Notre Dame, 255 Hurley Building, Notre Dame, IN 46556}
\email{nsalter@nd.edu}

\title{Noninjectivity of the monodromy of certain equicritical strata}

\date{November 4, 2024}

\begin{document}
\begin{abstract}
    An equicritical stratum is the locus of univariate monic squarefree complex polynomials where the critical points have prescribed multiplicities. Tracking the positions of both roots and critical points, there is a natural ``monodromy map'' taking the fundamental group into a braid group. We show here that when there are exactly two critical points, this monodromy map is noninjective.
\end{abstract}

\maketitle

\section{Introduction}

Let $\Poly_n(\C)$ denote the space of monic squarefree complex polynomials of degree $n$. Identifying monic polynomials with their roots induces an isomorphism with the space $\UConf_n(\C)$ of unordered $n$-tuples in $\C$. From the polynomial perspective, it is natural to consider the stratification on $\Poly_n(\C)$ induced by the multiplicities of the critical points of $f(z) \in \Poly_n(\C)$, or equivalently the multiplicities of the roots of $f'(z)$. For a partition $\kappa = \{k_1, \dots, k_r\}$ of $n-1$, define the {\em equicritical stratum} $\Poly_n(\C)[\kappa]$ as the locus in $\Poly_n(\C)$ consisting of those $f$ whose critical points have multiplicities specified by $\kappa$. 

Each $\Poly_n(\C)[\kappa]$ admits a natural map into the space $\UConf_{n+r}(\C)$ by associating to $f$ the $(n+r)$-tuple of roots and critical points (note that a root coincides with a critical point if and only if $f$ is not squarefree). Moreover, after lifting to a finite cover of $\UConf_{n+r}(\C)$ that distinguishes roots and critical points, this map is an embedding. Passing to the fundamental group, this induces a {\em monodromy homomorphism}
\[
\rho: \sB_n[\kappa] \to B_{n+r},
\]
where we have defined $\sB_n[\kappa]:= \pi_1(\Poly_n(\C)[\kappa])$ and $B_{n+r} = \pi_1(\UConf_{n+r}(\C))$. The latter is the usual braid group (here on $n+r$ strands), while we call the former a {\em stratified braid group}. 

Equicritical strata are simple examples of families of smooth varieties embedded in some fixed ambient space, subject to certain geometric constraints. Generally in such problems, it is of interest to understand how much of the topology of such a family is captured by various flavors of monodromy representation. There are at least two variants of monodromy one can study: the {\em intrinsic} monodromy, valued in some automorphism group associated to the variety itself (in our setting, this is just the induced permutation on the roots and critical points), as well as the finer {\em embedded} monodromy, which tracks the automorphism of the {\em pair} of the ambient space and the subvariety.

The importance of the embedded monodromy has been recognized by many authors in many problems, going at least as far back as the work of Dolgachev--Libgober \cite{DL} in the context of families of smooth degree-$d$ projective hypersurfaces. In spite of this, there have been only sporadic results. The purpose of this paper is to show that for a certain class of equicritical strata, namely those with exactly two critical points, the embedded monodromy $\rho$ is, somewhat to the authors' surprise, non-injective. 

\begin{theorem}\label{maintheorem}
For all $n\geq3$ and all $p,q\geq1$ such that $p+q=n-1$, the monodromy map
\[
  \rho \colon \sB_n[p,q] \to B_{n+2}
\]
is not injective.
\end{theorem}

\para{Context: equicritical strata} The equicritical strata were introduced by the second-named author in \cite{eqstrat1} (though are implicit in the work \cite{DM1,DM2} of Dougherty--McCammond). Subsequently in \cite{eqstrat2}, the image of $\rho$ in $B_{n+r}$ was completely determined. In the nontrivial regime $r > 1$, the image turns out to be a ``framed braid group'' --- a certain infinite-index subgroup of $B_{n+r}$ consisting of braids that preserve a certain ``winding number function'' assigning winding numbers to arcs on a punctured disk. Ultimately this constraint arises from the geometric interpretation of the logarithmic derivative of $f$ (a meromorphic differential on $\CP^1$) as a translation surface. In general the stratified braid groups $\sB_n[\kappa]$ themselves are rather mysterious, but can be understood as an extension of their image in $B_{n+r}$ by $\ker(\rho)$. Following the work of \cite{eqstrat2} which determines $\im(\rho)$, attention shifts to the problem of understanding $\ker(\rho)$. Note, however, that in the case $r=2$ studied here, the algebraic structure of $\sB_n[\kappa]$ is well-understood --- \Cref{prop:bundle} shows that $\sB_n[p,q] \cong F_2 \times \Z$ (in the case $p = q$, only after passing to a subgroup of index $2$). \\

\noindent\textbf{What about higher r?} It is not difficult to see that the explicit kernel elements constructed in \Cref{section:construction} generalize to kernel elements for general equicritical strata. The issue is that some new mechanism for certifying their nontriviality in $\sB_n[\kappa]$ must be developed. In the case $r = 2$, the derivative map $D$ studied in \Cref{section:bundle} realizes $\Poly_n(\C)[p,q]$ as the total space of a fiber bundle with fiber having free fundamental group. As observed in \Cref{remark:notbundle}, this no longer holds as soon as $r \ge 3$. It is not clear to the authors whether $\rho$ should be injective for $r \ge 3$.

\para{Outline of proof and organization of paper} In \Cref{section:bundle}, we show that, after possibly passing to a two-sheeted cover, $\Poly_n(\C)[p,q]$ can be identified with a {\em product} $\PConf_2(\C) \times (\C \setminus\{0,1\})$ (here and throughout, $\PConf_n(\C)$ denotes the configuration space of {\em ordered} $n$-tuples in $\C$), and describe a pair of explicit free generators $x,y$ for the fundamental group of the fiber. The monodromy image of $x,y$ is determined in \Cref{section:description}. Then in \Cref{section:construction}, we exhibit an explicit nontrivial word in $x,y$ and prove that it lies in the kernel of $\rho$.

\para{Acknowledgments} The second-named author gratefully acknowledges support from the National Science Foundation (grants no. DMS-2153879 and DMS-2338485). 

\section{A global description of $\Poly_n(\C)[p,q]$}\label{section:bundle}

We now fix $p,q\geq1$, and let $n=p+q+1$.

\subsection{Bundle structure}

\begin{definition}
  Let $\Poly_n(\C)[[p,q]]$ be the (one- or two-sheeted) connected cover of $\Poly_n(\C)[p,q]$ that equips a polynomial with an ordering of its two critical points.
\end{definition}

Note that this cover is two-sheeted if and only if $p=q$.

\begin{proposition}\label{prop:bundle}
  There is an isomorphism of varieties
  \[
  \Poly_n(\C)[[p,q]] \cong \PConf_2(\C) \times \left(\C \setminus \{0,1\}\right).
  \]
\end{proposition}

\begin{proof}
  An element of $\Poly_n(\C)[[p,q]]$ is determined by
  \begin{itemize}
  \item the ordered pair of its critical points,
  \item the constant term of the underlying polynomial.
  \end{itemize}

  This allows us to view $\Poly_n(\C)[[p,q]]$ as a Zariski-open subset of $\C^3$. To more easily describe what this subset actually is, we choose the following coordinates $(a,b,c)$ on $\C^3$, where:
  \begin{itemize}
  \item $a$ is the first critical point (of multiplicity $p$),
  \item $b$ is the second critical point (of multiplicity $q$),
  \item $c$ is the value of the underlying polynomial at $a$.
  \end{itemize}
  The polynomial associated to the triple $(a,b,c)$ is
  \[
    f_{(a,b,c)}(z) = n\int_a^z(u-a)^p(u-b)^q\,du + c.
  \]
  The critical points of $f_{(a,b,c)}$ are $a$ and $b$. Thus the critical values are $f_{(a,b,c)}(a)=c$ and
  \begin{align*}
    f_{(a,b,c)}(b) &= n\int_a^b(u-a)^p(u-b)^q\,du + c = n(b-a)^n\int_0^1v^p(v-1)^q\,dv + c = \mu (b-a)^n + c,
  \end{align*}
  where $\mu\neq0$ is the constant $n\int_0^1v^p(v-1)^q\,dv$. Therefore
  \begin{equation*}\label{eq:subset-of-c3}
    \Poly_n(\C)[[p,q]] = \C^3 - V(b-a) - V(c) - V(\mu (b-a)^n + c).
  \end{equation*}
    Making the change of coordinates $d = \frac{-c}{\mu (b-a)^n}$ gives the identification
    \[
    \Poly_n(\C)[[p,q]] = \C^3 - V(b-a) - V(d) - V(d-1),
    \]
    from which the identification
    \[
    \Poly_n(\C)[[p,q]] \cong \PConf_2(\C) \times \left(\C \setminus \{0,1\}\right)
    \]
    follows.
\end{proof}

\begin{remark}
  The projection map $D: \Poly_n(\C)[[p,q]] \to \PConf_2(\C)$ is a disguised form of the ordinary derivative, assigning as it does to $f \in \Poly_n(\C)[[p,q]]$ the ordered tuple of its critical points. The proof of \Cref{prop:bundle} shows that $D$ realizes $\Poly_n(\C)[[p,q]]$ as a globally trivial fiber bundle over $\PConf_2(\C)$ with fiber the twice-punctured plane. When $p = q$, $D$ descends to realize $\Poly_n(\C)[p,p]$ as a {\em nontrivial} bundle over $\UConf_2(\C)$ with fiber the twice-punctured plane. 
\end{remark}

\begin{remark}\label{remark:notbundle}
  For stratified configurations spaces for at least 3 critical points, the map $D$ recording the critical points is no longer a locally trivial fibration. For example, for the map $\Poly_4[1,1,1](\C)\to\UConf_3(\C)$ that records a polynomial's critical points, the fiber over a configuration in $\UConf_3(\C)$ is generically a three-times punctured complex plane, but it is a twice punctured complex plane if and only if the configuration has the form $\{u,v,(u+v)/2\}$ where $u\neq v$.
\end{remark}

\subsection{Explicit free generators}

\begin{definition}
  We define two elements $x,y\in\sB_n[p,q]$ as follows. Let $f(z)$ be the polynomial
  \begin{equation}\label{eq:f}
    f(z) = n\int_0^zu^p(u-1)^q\,du.
  \end{equation}
  Note that $f(z)-c\in\Poly_n(\C)[p,q]$ if and only if $c\in\C-\{0,\mu\}$, where $\mu\neq0$ is the constant $n\int_0^1u^p(u-1)^q\,du$. We consider elements of $\sB_n[p,q]$ that can be represented as loops of the form $f(z)-\gamma(t)$, where $\gamma(t)$ is a path in $\C-\{0,\mu\}$.

  Let $x,y\in\sB_n[p,q]$ be the elements corresponding to $\gamma(t)$ being a counter-clockwise circular loop around $0$, and around $\mu$, respectively, both based at $\mu/2$.
\end{definition}

Implicit in this definition we are taking $f-\mu/2$ for our base point of $\Poly_n(\C)[p,q]$. That is, we are identifying $\sB_n[p,q]$ with the fundamental group $\pi_1(\Poly_n(\C)[p,q],f-\mu/2)$.

\begin{lemma}\label{generate-free-group}
  The elements $x,y\in\sB_n[p,q]$ generate a free group $F_2$ of rank 2.
\end{lemma}

\begin{proof}
  Since $x$ and $y$ can be represented by loops of polynomials whose critical points $\{0,1\}$ remain fixed, they lift to loops in (the one- or two-sheeted cover) $\Poly_n(\C)[[p,q]]$. The isomorphism $\Poly_n(\C)[[p,q]]\cong\PConf_2(\C)\times(\C-\{0,1\})$ from \Cref{prop:bundle} gives an isomorphism
  \[
    \pi_1(\Poly_n(\C)[[p,q]],f-\mu/2) \cong \pi_1(\PConf_2(\C),(0,1))\times\pi_1(\C-\{0,1\},1/2),
  \]
  which is isomorphic to $\Z\times F_2$. The generators $x$ and $y$ lie entirely inside the $\pi_1(\C-\{0,1\},1/2)\cong F_2$ factor, and are represented by counter-clockwise loops around 0 and 1, respectively. In particular they generate this free group of rank 2.
\end{proof}

\section{Explicit description of the monodromy homomorphism}\label{section:description}

Let $x,y\in\sB_n[p,q]$ be as in the previous section.

\begin{proposition}\label{explicit-monodromy}
  The monodromy homomorphism $\rho\colon\sB_n[p,q]\to B_{n+2}$ satisfies
  \[
    \rho(x) = \sigma_1\sigma_2\cdots\sigma_{p+1}\sigma_1, \quad \rho(y) = \sigma_{p+2}\sigma_{p+3}\cdots\sigma_{p+q+2}\sigma_{p+2}
  \]
\end{proposition}

Technically, we have only defined the monodromy homomorphism $\rho$ up to conjugation in $B_{n+2}$. To define $x,y\in\sB_n[p,q]$, we chose to view $\sB_n[p,q]$ as the fundamental group based at $f-\mu/2$, i.e. $\pi_1(\Poly_n(\C)[p,q],f-\mu/2)$. Since $f-\mu/2$ has critical points at 0 and 1, to make sense of the monodromy homomorphism $\sB_n[p,q]\to B_{n+2}$, we must choose an identification of $B_{n+2}$ with the fundamental group $\pi_1(\UConf_{n+2}(\C),f^{-1}(\mu/2)\cup\{0,1\})$. We describe how to obtain the required identification in \Cref{identification}. In order to do this, and compute the monodromies $\rho(x)$ and $\rho(y)$, we first give a description of $f$ as a branched cover in \Cref{f-as-branched-cover}.

\subsection{Description of $f\colon\CP^1\to\CP^1$ as a branched cover}\label{f-as-branched-cover}

Construct a surface $S$ as follows. Begin with $n$ copies $S_1,\ldots,S_n$ of $\CP^1=\C\cup\{\infty\}$. In the first $p$ copies $S_1,\ldots,S_p$ cut a ``red'' slit between $0$ and $\infty$ along the interval $(-\infty,0)\subset\C$. In the last $q$ copies $S_{p+2},\ldots,S_{p+q+1}=S_n$ cut a ``blue'' slit between $1$ and $\infty$ along the interval $(1,\infty)\subset\C$. In the final remaining copy $S_{p+1}$ we cut both a red slit along $(-\infty,0)$ and a blue slit along $(1,\infty)$. See \Cref{pre-gluing}.

\begin{figure}
  \centering
  \begin{tikzpicture}
    % Define the radius of the circles and the spacing
    \def\r{1}
    \def\normalspace{2.2*\r}
    \def\dotspace{2.7*\r}
    
    % Add red lines to the first three circles
    \draw[red, ultra thick] (-\r, 0) -- (-\r/4, 0);
    \draw[red, ultra thick] (\dotspace-\r, 0) -- (\dotspace-\r/4, 0);
    \draw[red, ultra thick] (\dotspace + \normalspace-\r, 0) -- (\dotspace + \normalspace-\r/4, 0);
    
    % Add blue lines to the last three circles
    \draw[blue, ultra thick] (\dotspace + \normalspace+\r/4, 0) -- (\dotspace + \normalspace+\r, 0);
    \draw[blue, ultra thick] (\dotspace + 2*\normalspace+\r/4, 0) -- (\dotspace + 2*\normalspace+\r, 0);
    \draw[blue, ultra thick] (\dotspace + 2*\normalspace + \dotspace+\r/4, 0) -- (\dotspace + 2*\normalspace + \dotspace+\r, 0);

    % Draw the circles and labels
    \draw[color=lightgray, thick] (0, 0) circle (\r);
    \fill (-\r/4, 0) circle (2pt) node[below] {\tiny 0};
    \fill (\r/4, 0) circle (2pt) node[below] {\tiny 1};

    \draw[color=lightgray, thick] (\dotspace, 0) circle (\r);
    \fill (\dotspace-\r/4, 0) circle (2pt) node[below] {\tiny 0};
    \fill (\dotspace+\r/4, 0) circle (2pt) node[below] {\tiny 1};

    \draw[color=lightgray, thick] (\dotspace + \normalspace, 0) circle (\r);
    \fill (\dotspace + \normalspace-\r/4, 0) circle (2pt) node[below] {\tiny 0};
    \fill (\dotspace + \normalspace+\r/4, 0) circle (2pt) node[below] {\tiny 1};

    \draw[color=lightgray, thick] (\dotspace + 2*\normalspace, 0) circle (\r);
    \fill (\dotspace + 2*\normalspace-\r/4, 0) circle (2pt) node[below] {\tiny 0};
    \fill (\dotspace + 2*\normalspace+\r/4, 0) circle (2pt) node[below] {\tiny 1};

    \draw[color=lightgray, thick] (\dotspace + 2*\normalspace + \dotspace, 0) circle (\r);
    \fill (\dotspace + 2*\normalspace + \dotspace-\r/4, 0) circle (2pt) node[below] {\tiny 0};
    \fill (\dotspace + 2*\normalspace + \dotspace+\r/4, 0) circle (2pt) node[below] {\tiny 1};

    \node at (0, \r+0.3) {$S_1$};
    \node at (\dotspace, \r+0.3) {$S_p$};
    \node at (\dotspace + \normalspace, \r+0.3) {$S_{p+1}$};
    \node at (\dotspace + 2*\normalspace, \r+0.3) {$S_{p+2}$};
    \node at (\dotspace + 2*\normalspace + \dotspace, \r+0.3) {$S_n$};
    
    % Add ellipses between the circles
    \node at (\dotspace/2, 0) {\dots};
    \node at (\dotspace + 2*\normalspace + \dotspace/2, 0) {\dots};
    \node at (\dotspace/2,\r+0.3) {\dots};
    \node at (\dotspace + 2*\normalspace + \dotspace/2, \r+0.3) {\dots};
  \end{tikzpicture}
  \caption{Patches of the surfaces $S_1,\ldots,S_n$}\label{pre-gluing}
\end{figure}
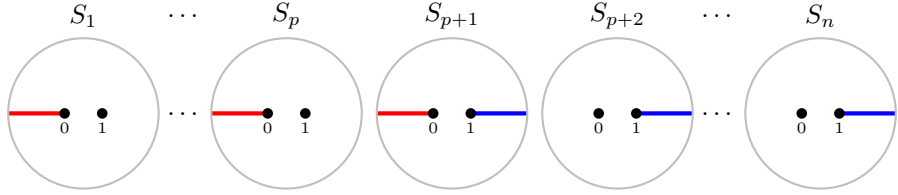

We now glue the surfaces $S_i$ as follows. For $1\leq i\leq p$, glue the bottom of the red slit of $S_i$ to the top of the red slit of $S_{i+1}$, and glue the bottom of the red slit of $S_{p+1}$ to the top of the red slit of $S_1$. For $p+1\leq i<n$, glue the top of the blue slit of $S_i$ to the bottom of the blue slit of $S_{i+1}$, and glue the top of the blue slit of $S_n$ to the bottom of the blue slit of $S_{p+1}$. Identify all copies of $\infty$ among $S_1,\ldots,S_n$, and separately identify all copies of $0$ among $S_1,\ldots,S_{p+1}$, and all copies of $1$ among $S_{p+1},\ldots,S_n$.

After doing all the gluing and identifications, we obtain a topological surface $S$ that is homeomorphic to the 2-sphere. The identifications of each $S_i$ with $\CP^1$ descend to a well-defined, degree $n$ branched covering map $S\to\CP^1$. This branched covering map induces a complex structure on $S$ by Riemann's existence theorem, and the resulting Riemann surface $S$ is isomorphic to $\CP^1$ by uniformization. We identify $S$ with $\CP^1$ via the unique isomorphism that sends the identified copy of $\infty$ from $S_1,\ldots,S_n$ to $\infty\in\CP^1$, the identified copy of $0$ from $S_1,\ldots,S_{p+1}$ to $0\in\CP^1$, and the identified copy of $1$ from $S_{p+1},\ldots,S_n$ to $1\in\CP^1$. We now view the original surfaces $S_1,\ldots,S_n$ as subsets of $\CP^1$. See \Cref{post-gluing} for a depiction of this away from $\infty$.

\begin{figure}
  \centering
  \begin{tikzpicture}
    \node at (0,0) {\includegraphics[scale=0.6]{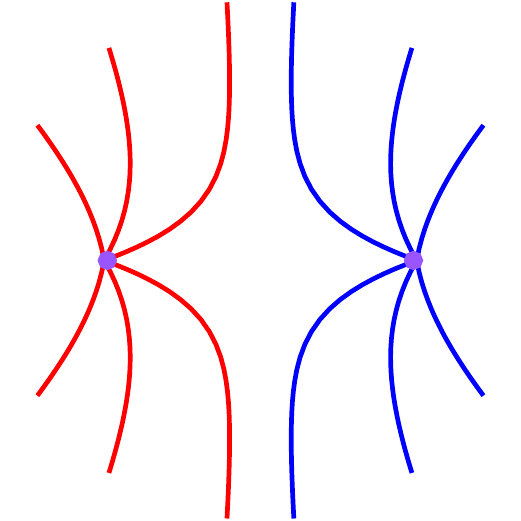}};
    \node at (-1.85,0) {\small $0$};
    \node at (1.85,0) {\small $1$};
    \node at (-0.85, -1.85) {\small $S_1$};
    \node at (-1.75, -1.35) {\small $S_2$};
    \node at (-1.8, 1.4) {\small $S_{p-1}$};
    \node at (-0.85, 1.85) {\small $S_p$};
    \node at (0,0) {\small $S_{p+1}$};
    \node at (0.85, 1.85) {\small $S_{p+2}$};
    \node at (1.77, 1.4) {\small $S_{p+3}$};
    \node at (1.8, -1.35) {\small $S_{n-1}$};
    \node at (0.85, -1.85) {\small $S_n$};

    \draw[color=lightgray, very thick] (0,0) circle (2.66);

    \fill[color=blue] (2.25, -0.4) circle (1pt);
    \fill[color=blue] (2.3, 0.0) circle (1pt);
    \fill[color=blue] (2.25, 0.4) circle (1pt);

    \fill[color=red] (-2.25, -0.4) circle (1pt);
    \fill[color=red] (-2.3, 0.0) circle (1pt);
    \fill[color=red] (-2.25, 0.4) circle (1pt);
  \end{tikzpicture}
  \caption{A patch of the glued surface $S\cong\CP^1$}\label{post-gluing}
\end{figure}
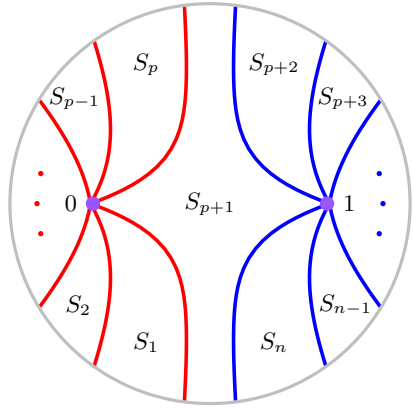

This identification of $S$ with $\CP^1$ identifies the degree $n$ branched cover $S\to\CP^1$ with a rational map $g\colon\CP^1\to\CP^1$. Since $g^{-1}(\infty)=\{\infty\}$, it follows that $g$ is a polynomial. The derivative $g'(z)$ of this polynomial is a scalar multiple of $f'(z)$, and $g(0)=0=f(0)$, while $g(1)=1$ and $f(1)=\mu$. Hence $f=\mu g$.

\subsection{Identification with the braid group}\label{identification}

Let $m\geq2$ and $X\in\UConf_m(\C)$. Consider an arc $\delta$ in $\C$, whose endpoints are two distinct elements of $X$, and whose interior does not meet $X$. The \emph{half twist} along $\delta$ is the element of $\pi_1(\UConf_m(\C),X)$ represented by a loop that fixes all elements of $S$ except the endpoints of $\delta$, which are swapped in a counterclockwise fashion, each endpoint traveling halfway around the boundary of a small thickening of $\delta$. Note that the definition of a half twist does not depend on any orientation $\delta$ may have, and only depends on $\delta$ up to homotopies that fix its endpoints.

To specify an isomorphism $B_m\cong\pi_1(\UConf_m(\C),X)$, it suffices to specify a collection of arcs $\delta_1,\ldots,\delta_{m-1}$ that connect pairs of points $\{x_1,x_2\},\{x_2,x_3\},\ldots,\{x_{m-1},x_m\}$, where $X=\{x_1,\ldots,x_m\}$. The interiors of these arcs must also avoid the points in the configuration $X$. Then the standard generators $\sigma_1,\ldots,\sigma_{m-1}$ of $B_m$ will correspond to the half twists about $\delta_1,\ldots,\delta_{m-1}$.

In our setting, we have $m=n+2$, and $X=f^{-1}(\mu/2)\cup\{0,1\}$. For each $1\leq i\leq n$, let $z_i$ be the unique point in $S_i$ such that $f(z_i)=\mu/2$, so that $X=\{z_1,\ldots,z_n,0,1\}$.

To specify arcs of the sort mentioned above, we will use the branched cover $f\colon\C\to\C$ to lift paths in $\C$ whose endpoints lie in $\{0,\mu/2,\mu\}$, to paths in $\C$ whose endpoints lie in $\{z_1,\ldots,z_n,0,1\}$. Choosing arcs of this form will allow us to easily read off the values of $\rho(x)$ and $\rho(y)$ in the braid group $B_{n+2}$.

\begin{definition}
  Define arcs $\alpha_1,\ldots,\alpha_n$, $\beta_1,\ldots,\beta_n$, $\ell_1^0,\ldots,\ell_n^0$, and $\ell_1^1,\ldots,\ell_n^1$ in $\C$ as follows. For each $1\leq i\leq n$, the arcs $\alpha_i,\beta_i,\ell_i^0,\ell_i^1$ all end at $z_i\in S_i$. Furthermore, they are all obtained by using $f$ to lift paths as follows:
  \begin{itemize}
  \item $\alpha_i$ is a lift of the counterclockwise circular loop centered at $0$, based at $\mu/2$.
  \item $\beta_i$ is a lift of the counterclockwise circular loop centered at $\mu$, based at $\mu/2$.
  \item $\ell_i^0$ is a lift of the line segment connecting $0$ and $\mu/2$.
  \item $\ell_i^1$ is a lift of the line segment connecting $1$ and $\mu/2$.
  \end{itemize}
\end{definition}

By definition of $x$ and $y$, the braids $\rho(x)$ and $\rho(y)$ are represented by the loops $\{\alpha_1,\ldots,\alpha_n,0,1\}$ and $\{\beta_1,\ldots,\beta_n,0,1\}$ in $\UConf_{n+2}(\C)$.

\begin{figure}
  \begin{minipage}[c]{0.47\linewidth}
  \hspace{0.13\linewidth}
  \begin{tikzpicture}
    \node at (0,0) {\includegraphics[scale=0.6]{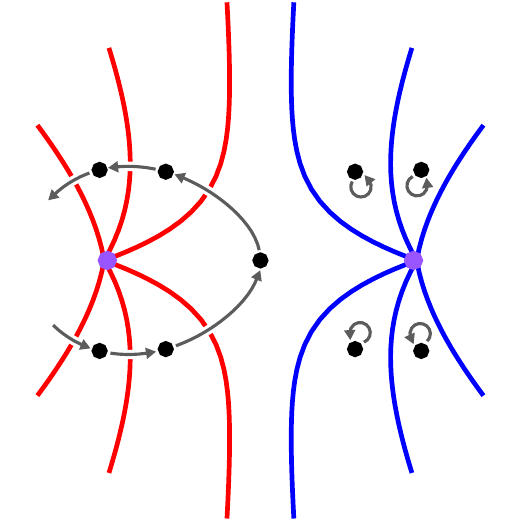}};
    \node at (-1.8,0) {\tiny $0$};
    \node at (1.8,0) {\tiny $1$};
    \node at (-0.85, -1.15) {\tiny $z_1$};
    \node at (-1.71, -1.16) {\tiny $z_2$};
    \node at (-1.71, 1.16) {\tiny $z_{p-1}$};
    \node at (-0.85, 1.12) {\tiny $z_p$};
    \node at (-0.4, -0.05) {\tiny $z_{p+1}$};
    \node at (0.85, 1.12) {\tiny $z_{p+2}$};
    \node at (1.71, 1.16) {\tiny $z_{p+3}$};
    \node at (1.71, -1.16) {\tiny $z_{n-1}$};
    \node at (0.85, -1.15) {\tiny $z_n$};

    \draw[color=lightgray, very thick] (0,0) circle (2.66);

    \fill[color=blue] (2.25, -0.4) circle (1pt);
    \fill[color=blue] (2.3, 0.0) circle (1pt);
    \fill[color=blue] (2.25, 0.4) circle (1pt);

    \fill[color=red] (-2.25, -0.4) circle (1pt);
    \fill[color=red] (-2.3, 0.0) circle (1pt);
    \fill[color=red] (-2.25, 0.4) circle (1pt);
  \end{tikzpicture}
  \captionsetup{width=\linewidth}
  \caption{The monodromy $\rho(x)$}\label{rho-x-monodromy}
  \end{minipage}
  \hfill
  \begin{minipage}[c]{0.47\linewidth}
  \hspace{0.13\linewidth}
  \begin{tikzpicture}
    \node at (0,0) {\includegraphics[scale=0.6]{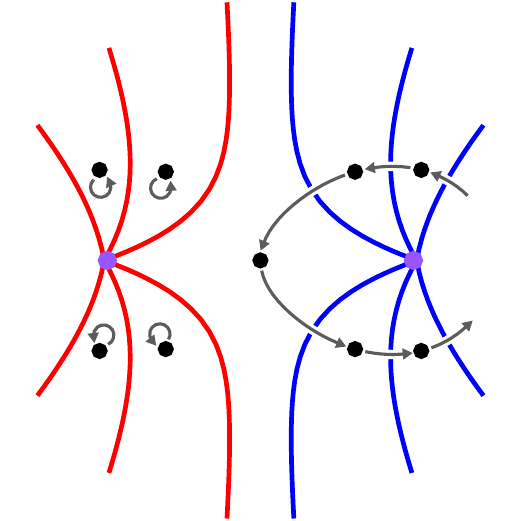}};
    \node at (-1.8,0) {\tiny $0$};
    \node at (1.8,0) {\tiny $1$};
    \node at (-0.85, -1.15) {\tiny $z_1$};
    \node at (-1.71, -1.16) {\tiny $z_2$};
    \node at (-1.71, 1.16) {\tiny $z_{p-1}$};
    \node at (-0.85, 1.12) {\tiny $z_p$};
    \node at (-0.4, -0.05) {\tiny $z_{p+1}$};
    \node at (0.85, 1.12) {\tiny $z_{p+2}$};
    \node at (1.71, 1.16) {\tiny $z_{p+3}$};
    \node at (1.71, -1.16) {\tiny $z_{n-1}$};
    \node at (0.85, -1.15) {\tiny $z_n$};

    \draw[color=lightgray, very thick] (0,0) circle (2.66);

    \fill[color=blue] (2.25, -0.4) circle (1pt);
    \fill[color=blue] (2.3, 0.0) circle (1pt);
    \fill[color=blue] (2.25, 0.4) circle (1pt);

    \fill[color=red] (-2.25, -0.4) circle (1pt);
    \fill[color=red] (-2.3, 0.0) circle (1pt);
    \fill[color=red] (-2.25, 0.4) circle (1pt);
  \end{tikzpicture}
  \captionsetup{width=\linewidth}
  \caption{The monodromy $\rho(y)$}\label{rho-y-monodromy}
  \end{minipage}
\end{figure}

From the description of $f\colon\C\to\C$ as a branched cover, the following topological facts are immediate:
\begin{itemize}
\item The arcs $\alpha_1,\ldots,\alpha_{p+1}$ form a single loop connecting the points $z_1,\ldots,z_{p+1}$, and each of these arcs only meet two of the $S_j$. For $p+2\leq i\leq n$, the arc $\alpha_i$ is a loop based at $z_i$ completely contained in $S_i$. See \Cref{rho-x-monodromy}.
\item The arcs $\beta_{p+1},\ldots,\beta_n$ form a single loop connecting the points $z_{p+1},\ldots,z_n$, and each of these arcs only meet two of the $S_j$. For $1\leq i\leq p$, the arc $\beta_i$ is a loop based at $z_i$ that is completely contained in $S_i$. See \Cref{rho-y-monodromy}.
\item For $1\leq i\leq p+1$, the arc $\ell_i^0$ connects $z_i$ and $0$, and the interior is completely contained in $S_i$.
\item For $p+1\leq i\leq n$, the arc $\ell_i^1$ connects $z_i$ and $1$, and the interior is completely contained in $S_i$.
\end{itemize}
In particular, we can represent $\rho(x)$ by the path $\{\alpha_1,\ldots,\alpha_{p+1},z_{p+2},\ldots,z_n,0,1\}$, and $\rho(y)$ by the path $\{z_1,\ldots,z_p,\beta_{p+1},\ldots,\beta_n,0,1\}$.

The sequence of arcs that we used to identify $\pi_1(\UConf_{n+2}(\C),\{z_1,\ldots,z_n,0,1\})$ with $B_{n+2}$ is as follows (see \Cref{sequence-of-arcs}):
\[
  \ell_1^0, \ell_2^0, \alpha_2, \ldots, \alpha_p, \ell_{p+1}^1, \ell_{p+2}^1, \beta_{p+2}, \ldots, \beta_{n-1}.
\]

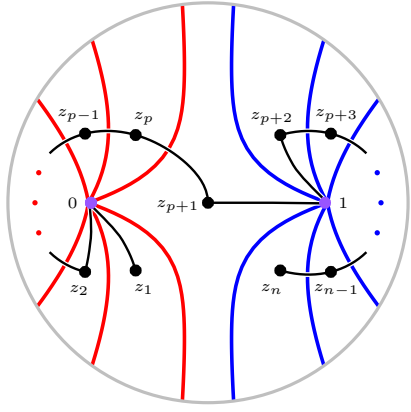
\begin{figure}
  \centering
  \begin{tikzpicture}
    \node at (0,0) {\includegraphics[scale=0.6]{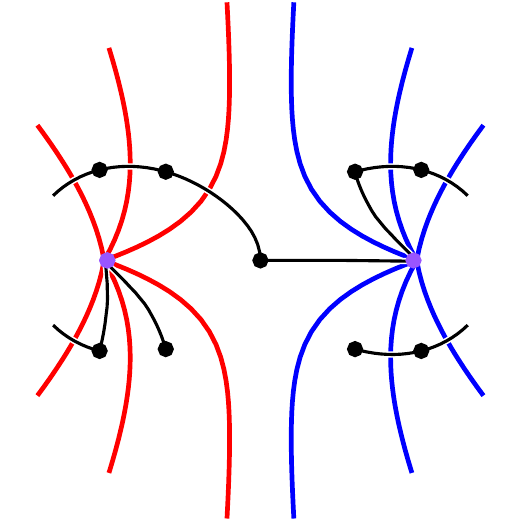}};
    \node at (-1.8,0) {\tiny $0$};
    \node at (1.8,0) {\tiny $1$};
    \node at (-0.85, -1.15) {\tiny $z_1$};
    \node at (-1.71, -1.16) {\tiny $z_2$};
    \node at (-1.71, 1.16) {\tiny $z_{p-1}$};
    \node at (-0.85, 1.12) {\tiny $z_p$};
    \node at (-0.4, -0.05) {\tiny $z_{p+1}$};
    \node at (0.85, 1.12) {\tiny $z_{p+2}$};
    \node at (1.71, 1.16) {\tiny $z_{p+3}$};
    \node at (1.71, -1.16) {\tiny $z_{n-1}$};
    \node at (0.85, -1.15) {\tiny $z_n$};

    \draw[color=lightgray, very thick] (0,0) circle (2.66);

    \fill[color=blue] (2.25, -0.4) circle (1pt);
    \fill[color=blue] (2.3, 0.0) circle (1pt);
    \fill[color=blue] (2.25, 0.4) circle (1pt);

    \fill[color=red] (-2.25, -0.4) circle (1pt);
    \fill[color=red] (-2.3, 0.0) circle (1pt);
    \fill[color=red] (-2.25, 0.4) circle (1pt);
  \end{tikzpicture}
  \captionsetup{width=\linewidth}
  \caption{The sequence of arcs $\ell_1^0, \ell_2^0, \alpha_2, \ldots, \alpha_p, \ell_{p+1}^1, \ell_{p+2}^1, \beta_{p+2}, \ldots, \beta_{n-1}$.}\label{sequence-of-arcs}
\end{figure}

Implicit in this sequence of arcs is the ordering $z_1,0,z_2,\ldots,z_{p+1},1,z_{p+2},\ldots,z_n$ of the elements of $\{z_1,\ldots,z_n,0,1\}$.

From our description of $\rho(x)$ and $\rho(y)$, as well as our observations about $\ell_i^0$ and $\ell_i^1$, it follows that with respect to our identification of $\pi_1(\UConf_{n+2}(\C),\{z_1,\ldots,z_n,0,1\})$ with $B_{n+2}$, we do indeed have
\[
  \rho(x) = \sigma_1\cdots\sigma_{p+1}\sigma_1, \qquad \rho(y) = \sigma_{p+2}\cdots\sigma_{n+1}\sigma_{p+2},
\]
as required. See \Cref{rho-x-braid,rho-y-braid}. This completes the proof of \Cref{explicit-monodromy}. \qed

\begin{figure}
  \centering
  \begin{minipage}[c]{0.48\linewidth}
    \begin{center}
    \begin{tikzpicture}
      \node at (0,0) {\includegraphics[width=0.73\textwidth]{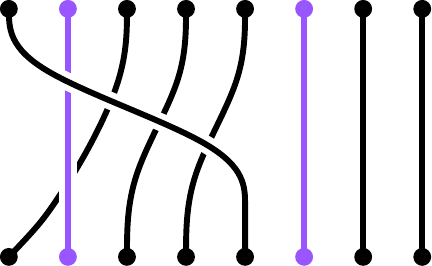}};
      \node at (-2.8,2) {\small $z_1$};
      \node at (-2,2.025) {\small $0$};
      \node at (-1.2,2) {\small $z_2$};
      \node at (-0.4,2) {\small $z_p$};
      \node at (0.4,2) {\small $z_{p+1}$};
      \node at (1.2,2.025) {\small $1$};
      \node at (2,2) {\small $z_{p+2}$};
      \node at (2.8,2) {\small $z_n$};
    \end{tikzpicture}
    \end{center}
    \captionsetup{width=0.9\linewidth,justification=centering}
    \caption{\\$\rho(x)\in B_{p+q+3}$ for $p=3$ and $q=2$.}\label{rho-x-braid}
  \end{minipage}
  \begin{minipage}[c]{0.48\linewidth}
    \begin{center}
    \begin{tikzpicture}
      \node at (0,0) {\includegraphics[width=0.73\textwidth]{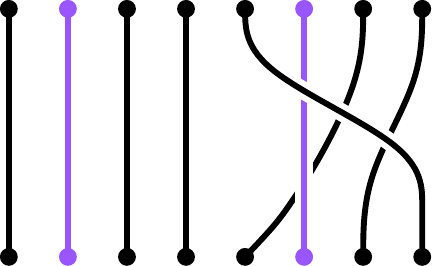}};
      \node at (-2.8,2) {\small $z_1$};
      \node at (-2,2.025) {\small $0$};
      \node at (-1.2,2) {\small $z_2$};
      \node at (-0.4,2) {\small $z_p$};
      \node at (0.4,2) {\small $z_{p+1}$};
      \node at (1.2,2.025) {\small $1$};
      \node at (2,2) {\small $z_{p+2}$};
      \node at (2.8,2) {\small $z_n$};
    \end{tikzpicture}
    \end{center}
    \captionsetup{width=0.9\linewidth,justification=centering}
    \caption{\\$\rho(y)\in B_{p+q+3}$ for $p=3$ and $q=2$.}\label{rho-y-braid}
  \end{minipage}
\end{figure}

\section{Construction of non-trivial elements in the kernel}\label{section:construction}

The following theorem establishes \Cref{maintheorem} by explicitly constructing a non-trivial element of $\sB_n[p,q]$ with trivial monodromy.

\begin{theorem}\label{kernel-element}
  The word
  \[
    [x^{p+1}, yx^py^{-1}xy]\in\sB_n[p,q]
  \]
  is a non-trivial element of $\ker(\rho\colon\sB_n[p,q]\to B_{n+2})$.
\end{theorem}

The non-triviality of this word follows immediately from \Cref{generate-free-group}: the centralizer of $x^{p+1}$ in the free group on $x$ and $y$ is $\langle x\rangle$, and $yx^py^{-1}xy$ is not a power of $x$. The main content of \Cref{kernel-element} is that the image of this word is trivial in $B_{n+2}$. To prove \Cref{kernel-element} we will write $\rho(yx^py^{-1}xy)$ as a product of braids that all commute with $\rho(x^{p+1})$. To facilitate writing this product, we begin by defining some special elements of the braid group.

\subsection{Products of consecutive standard generators}

Let $m\geq2$. For $1\leq i,j<m$ we define the following elements of the braid group $B_m$.
\[
  \sigma_{[i,j]} \coloneqq \begin{cases} \sigma_i\cdots\sigma_j & i \leq j \\ 1 & i > j, \end{cases} \qquad \bar{\sigma}_{[i,j]} \coloneqq \begin{cases} \sigma_j\cdots\sigma_i & i \leq j \\ 1 & i > j. \end{cases}
\]
\begin{figure}
  \centering
  \hspace{0.09\textwidth}\includegraphics[width=0.3\textwidth]{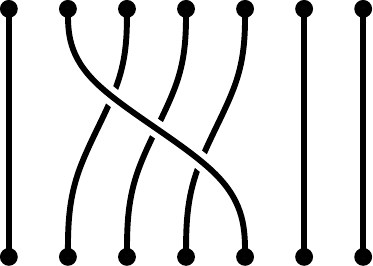}\hfill
  \includegraphics[width=0.3\textwidth]{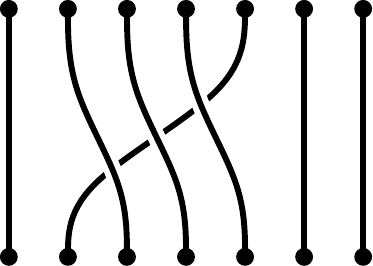}\hspace{0.075\textwidth}
  \caption{The braids $\sigma_{[2,4]}$ (left) and $\bar{\sigma}_{[2,4]}$ (right) in $B_7$.}\label{consecutive-sigma-product}
\end{figure}
See \Cref{consecutive-sigma-product} for an example depicting these braids. The commuting relations in the braid group imply that for $1\leq i,j,i',j'<m$
\[
  \sigma_{[i,j]}\sigma_{[i',j']} = \sigma_{[i',j']}\sigma_{[i,j]}, \quad \text{if } |i'-j|\geq2 \text{ or } |i-j'|\geq2.
\]

For $1\leq i<m-1$, the braid relation $\sigma_i\sigma_{i+1}\sigma_i=\sigma_{i+1}\sigma_i\sigma_{i+1}$ can be expressed by the identity
\[
  \sigma_{[i,i+1]}\sigma_i = \sigma_{i+1}\sigma_{[i,i+1]}.
\]
In fact, if $1\leq i<j<m$, then we have
\[
  \sigma_{[i,j]}\sigma_i = \sigma_{[i,i+1]}\sigma_{[i+2,j]}\sigma_i = \sigma_{[i,i+1]}\sigma_i\sigma_{[i+2,j]} = \sigma_{i+1}\sigma_{[i,i+1]}\sigma_{[i+2,j]} = \sigma_{i+1}\sigma_{[i,j]}.
\]

\subsection{Dehn twists}

Let $2\leq k\leq m$ and define the following element of $PB_m$:
\[
  \twist_k \coloneqq (\sigma_{[1,k-1]}\sigma_1)^{k-1}.
\]
\begin{figure}
  \begin{minipage}[c]{0.48\linewidth}
    \begin{center}
    \includegraphics[width=0.625\textwidth]{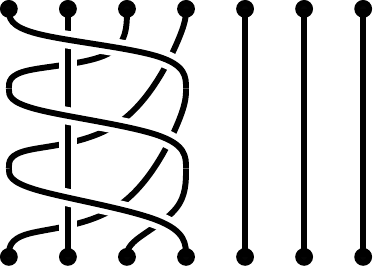}
    \captionsetup{width=\linewidth}
    \caption{Illustration of $\twist_4$ in $B_7$.}\label{dehn-twist}
    \end{center}
  \end{minipage}
  \hfill
  \begin{minipage}[c]{0.48\linewidth}
    \begin{center}
    \includegraphics[width=0.68\textwidth]{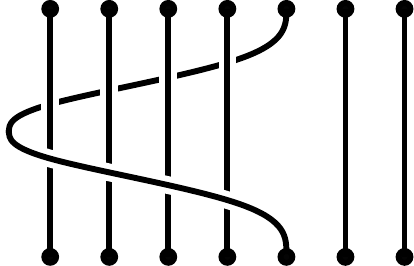}
    \captionsetup{width=\linewidth}
    \caption{Illustration of $\push_5$ in $B_7$.}\label{point-push-braid}
    \end{center}
  \end{minipage}
\end{figure}
See \Cref{dehn-twist} for an example of this braid. If we view $B_m$ as a mapping class group of an $m$-times punctured closed disk $D_m$, then $\twist_k$ can be represented by a Dehn twist about the boundary of a disk $D_k$ containing the first $k$ punctures. This disk $D_k$ is such that the inclusion $D_k\subseteq D_m$ induces the standard inclusion homomorphism $B_k\to B_m$.

Therefore, $\twist_k$ commutes with each of $\sigma_1,\ldots,\sigma_{k-1}$ and $\sigma_{k+1},\ldots,\sigma_{m-1}$.

\subsection{Point pushes}

Let $2\leq k\leq m$, and define the following element of $PB_m$:
\[
  \push_k \coloneqq \bar{\sigma}_{[1,k-1]}\sigma_{[1,k-1]} = (\sigma_{k-1}\cdots\sigma_1)(\sigma_1\cdots\sigma_{k-1}).
\]
See \Cref{point-push-braid} for an example of this braid. Once again view $B_m$ as a mapping class group of an $m$-times punctured closed disk $D_m$, and let $D_{k-1}\subset D_k\subseteq D_m$ be inclusions that induce the standard inclusion homomorphisms $B_{k-1}\to B_k\to B_m$. Define the closed annulus $A_k\coloneqq D_k-\interior(D_{k-1})$. Then $\push_k$ is the image of a counterclockwise generator of $\pi_1(A_k,p_k)$ under the point-pushing homomorphism $\pi_1(A_k,p_k)\to PB_m$. See \Cref{disk-push-diagram}.

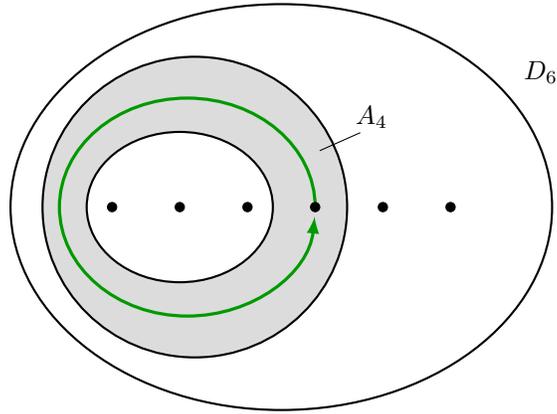
\begin{figure}
  \centering
  \begin{tikzpicture}
    \def\r{3}
    \def\s{3/4}
    \draw[thick] (0,0) ellipse ({1.2*\r} and {0.9*\r});

    \draw[thick, fill={rgb,255:red,220; green,220; blue,220}] ({-0.25-2/5*(\r-\s)},0) ellipse ({0.9*(\r - \s)} and 2);
    \draw[thick, fill=white] ({-3/5*(\r-\s)},0) ellipse ({2.75/5*(\r - \s)} and 1);

    \foreach \i in {0,...,5} {
      \node (P\i) at ({-3 + \s + 2/5*(\r-\s)*\i}, 0) {};
    }

    \draw [-latex, very thick, color={rgb,255:red,0; green,153; blue,0}] (P3) arc [start angle=0, end angle=355, x radius=1.7, y radius=1.45];

    \foreach \i in {0,...,5} {
      \draw[fill=black] (P\i) circle(1.75pt);
    }

    \node at (1.15*\r, 0.6*\r) {$D_6$};
    \node at (0.4*\r, 0.4*\r) {$A_4$};
    \draw (0.35*\r, 0.33*\r) -- (0.17*\r, 0.25*\r);
\end{tikzpicture}

  \caption{A path in $A_4\subset D_4\subset D_6$ representing the generator of $\pi_1(A_4,p_4)$ corresponding to $\push_4$ in $B_6$}
  \label{disk-push-diagram}
\end{figure}

Therefore, $\push_k$ commutes with each of $\sigma_1,\ldots,\sigma_{k-2}$ and $\sigma_{k+1},\ldots,\sigma_{m-1}$. We also have the following identity in $PB_m$:
\[
  \twist_{k+1} = \twist_k\cdot\push_{k+1}.
\]
\begin{figure}
  \centering
  \hspace{0.115\textwidth}\includegraphics[width=0.3\textwidth]{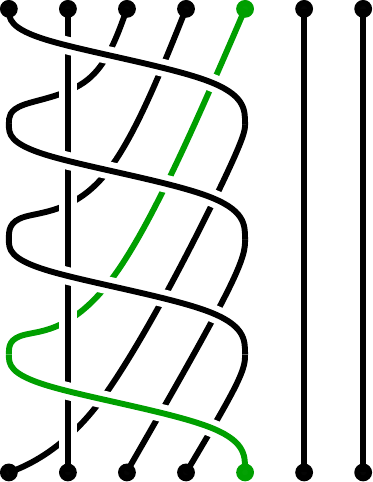}\hfill
  \includegraphics[width=0.32\textwidth]{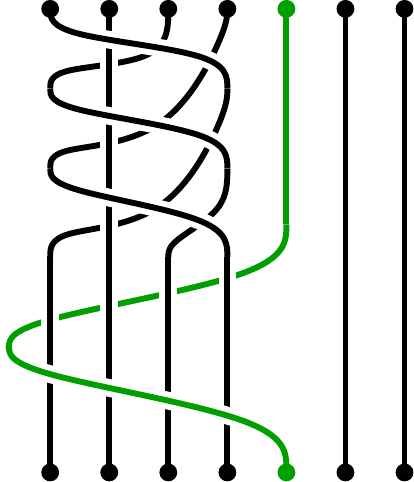}\hspace{0.115\textwidth}
  \caption{The braids $\twist_5$ (left) and $\twist_4\cdot\push_5$ (right) in $B_7$. The fifth strand is colored green only for visual clarity.}\label{twist-and-push-identity}
\end{figure}
See \Cref{twist-and-push-identity} for a depiction of this identity. To justify this identity, note that each term is supported on the once-marked annulus $A_k$. The identity asserts that the outer Dehn twist $\twist_{k+1}$ is equal to the inner Dehn twist $\twist_k$ composed with the counter-clockwise point push $\push_k$. It is not hard to see that $\twist_k\cdot\push_k\cdot\twist_{k+1}^{-1}$ may be represented by a full rotation of the annulus, i.e.\ the identity. Alternatively, the identity is easy to verify directly in the braid group when $k=2$, and the above description implies that its truth does not depend on the value of $k$.

\subsection{Rewriting $\rho(yx^py^{-1}xy)$}

We are now in a position to express $\rho(yx^py^{-1}xy)$ as a product of braids that all commute with $\rho(x^{p+1})$.

\begin{lemma}\label{magic-word-identity}
  The following identity holds:
  \[
    \rho(yx^py^{-1}xy) = \twist_{p+1}\sigma_1^{-1}\sigma_{p+1}^{-1}\sigma_1\sigma_{p+3}\push_{p+3}\sigma_{[p+3,n+1]}.
  \]
\end{lemma}
See \Cref{complicated-identity} for an example of this identity. Note that the factor $\sigma_1^{-1}\sigma_{p+1}^{-1}\sigma_1$ simplifies to $\sigma_{p+1}^{-1}$ if and only if $p\geq2$.
\begin{figure}
  \centering
  \hspace{0.075\textwidth}\includegraphics[width=0.35\textwidth]{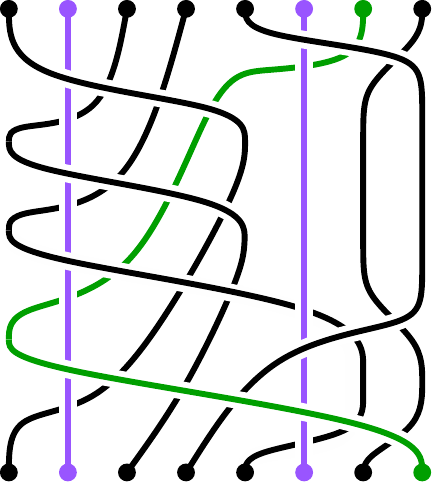}\hfill
  \includegraphics[width=0.38\textwidth]{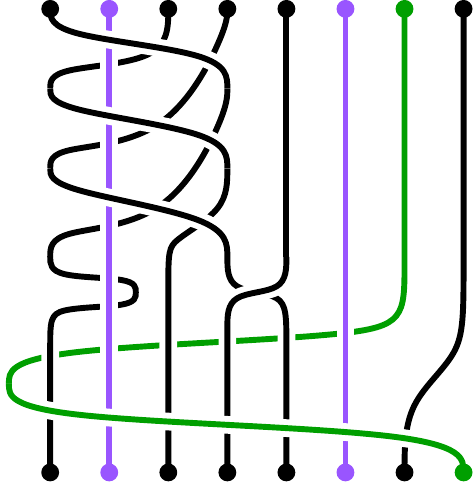}\hspace{0.075\textwidth}
  \caption{The left and right sides of the identity from \Cref{magic-word-identity} when $p=3$ and $q=2$. The second and sixth strands corresponding to critical points are colored purple, whereas the seventh strand is colored green only for visual clarity.}\label{complicated-identity}
\end{figure}

\begin{proof}
We will compute $\rho(yx^px^{-1}y^{-1})$ in chunks, and combine everything at the end. By \Cref{explicit-monodromy} we have
\begin{align*}
  \rho(x) &= \sigma_{[1,p+1]}\sigma_1 = \sigma_2\sigma_{[1,p+1]}, \\
  \rho(y) &= \sigma_{[p+2,n+1]}\sigma_{p+2} = \sigma_{p+3}\sigma_{[p+2,n+1]}.
\end{align*}

Hence,
\begin{align*}
  \rho(yx^py^{-1}) &= (\sigma_{p+3}\sigma_{[p+2,n+1]})\rho(x^p)(\sigma_{p+3}\sigma_{[p+2,n+1]})^{-1} \\
  &= (\sigma_{p+3}\sigma_{p+2})\rho(x^p)(\sigma_{p+3}\sigma_{p+2})^{-1},
\end{align*}

where the last equality holds because $\rho(x)\in B_{p+2}$ commutes with $\sigma_{[p+3,n+1]}$. We also have
\begin{align*}
  (\sigma_{p+3}\sigma_{p+2})^{-1}\rho(x) &= \sigma_{p+2}^{-1}\sigma_{p+3}^{-1}\sigma_{[1,p+1]}\sigma_1 \\ &= \sigma_{[1,p]}\sigma_{p+2}^{-1}\sigma_{p+1}\sigma_1\sigma_{p+3}^{-1} \\
                                   &= \sigma_{[1,p]}(\sigma_{p+1}\sigma_{p+2}\sigma_{p+1}^{-1}\sigma_{p+2}^{-1})\sigma_1\sigma_{p+3}^{-1} \\
                                   &= \sigma_{[1,p+1]}\sigma_{p+2}\sigma_{p+1}^{-1}\sigma_1\sigma_{p+2}^{-1}\sigma_{p+3}^{-1} \\
                                   &= \rho(x)\sigma_1^{-1}(\sigma_{p+2}\sigma_{p+1}^{-1}\sigma_1\sigma_{p+2}^{-1}\sigma_{p+3}^{-1}) \\
                                   &= \rho(x)\sigma_{p+2}(\sigma_1^{-1}\sigma_{p+1}^{-1}\sigma_1)(\sigma_{p+2}^{-1}\sigma_{p+3}^{-1}).
\end{align*}
Furthermore,
\begin{align*}
  (\sigma_{p+2}^{-1}\sigma_{p+3}^{-1})\rho(y)=\sigma_{p+2}^{-1}\sigma_{p+3}^{-1}\sigma_{p+3}\sigma_{[p+2,n+1]} = \sigma_{[p+3,n+1]}.
\end{align*}
We will also use the fact that
\[
  \rho(x^{p+1}) = \twist_{p+2}.
\]

Therefore,
\begin{align*}
  \rho(yx^py^{-1}xy) &= (\sigma_{p+3}\sigma_{p+2})\rho(x^p)(\sigma_{p+3}\sigma_{p+2})^{-1}\rho(x)\rho(y) \\
               &= (\sigma_{p+3}\sigma_{p+2})\rho(x^{p+1})\sigma_{p+2}(\sigma_1^{-1}\sigma_{p+1}^{-1}\sigma_1)(\sigma_{p+2}^{-1}\sigma_{p+3}^{-1})\rho(y) \\
               &= (\sigma_{p+3}\sigma_{p+2})\rho(x^{p+1})\sigma_{p+2}(\sigma_1^{-1}\sigma_{p+1}^{-1}\sigma_1)\sigma_{[p+3,n+1]} \\
               &= (\sigma_{p+3}\sigma_{p+2})\twist_{p+2}\sigma_{p+2}(\sigma_1^{-1}\sigma_{p+1}^{-1}\sigma_1)\sigma_{[p+3,n+1]} \\ &= (\sigma_{p+3}\sigma_{p+2})\twist_{p+1}\push_{p+2}\sigma_{p+2}(\sigma_1^{-1}\sigma_{p+1}^{-1}\sigma_1)\sigma_{[p+3,n+1]} \\
                     &= \sigma_{p+3}\twist_{p+1}\push_{p+3}(\sigma_1^{-1}\sigma_{p+1}^{-1}\sigma_1)\sigma_{[p+3,n+1]} \\
                     &= \twist_{p+1}\sigma_1^{-1}\sigma_{p+1}^{-1}\sigma_1\sigma_{p+3}\push_{p+3}\sigma_{[p+3,n+1]}.\qedhere
\end{align*}
\end{proof}

\subsection{Proof that $\rho(x^{p+1})$ and $\rho(yx^py^{-1}xy)$ commute}

\begin{proof}
  It was observed earlier that $\twist_{p+2}$ commutes with $\sigma_1,\ldots,\sigma_{p+1}$ and $\sigma_{p+3},\ldots,\sigma_{n+1}$. We also observed that $\push_{p+3}$ commutes with everything in $B_{p+2}$, including $\twist_{p+2}$. Hence $\rho(x^{p+1})=\twist_{p+2}$ commutes with $\rho(yx^py^{-1}xy)$ by \Cref{magic-word-identity}.
\end{proof}

This completes the proof of \Cref{kernel-element}, and hence also \Cref{maintheorem}.

\bibliographystyle{alpha}
\bibliography{references}

\end{document}